\documentclass[12pt,a4paper]{article}
\usepackage{indentfirst}
\usepackage{amsmath,amssymb,amsfonts,amsthm,graphics}
\usepackage{makeidx}
\usepackage{ifpdf}
\newtheorem{thm}{Theorem}

\newtheorem{lem}{Lemma}

\theoremstyle{definition}

\def\-{\mbox{--}}

\newtheorem{obs}{Observation}

\begin{document}
\title{Characterize graphs with rainbow connection number $m-2$ and $m-3$ \footnote{Supported by NSFC, PCSIRT and the ``973"
program. }}
\author{
\small  Xueliang Li, Yuefang Sun, Yan Zhao\\
\small Center for Combinatorics and LPMC-TJKLC\\
\small Nankai University, Tianjin 300071, P.R. China\\
\small E-mails: lxl@nankai.edu.cn, syf@cfc.nankai.edu.cn, zhaoyan2010@mail.nankai.edu.cn}
\date{}
\maketitle
\begin{abstract}
A path in an edge-colored graph, where adjacent edges may be colored
the same, is a rainbow path if no two edges of it are colored the
same. A nontrivial connected graph $G$ is rainbow connected if there
is a rainbow path connecting any two vertices, and the rainbow
connection number of $G$, denoted by $rc(G)$, is the minimum number of
colors that are needed in order to make $G$ rainbow connected.
Chartrand et al. obtained that $G$ is a tree if and only if
$rc(G)=m$, and it is easy to see that $G$ is not a tree if and only
if $rc(G)\leq m-2$, where $m$ is the number of edge of $G$. So there
is an interesting problem: Characterize the graphs $G$ with
$rc(G)=m-2$. In this paper, we settle down this problem. Furthermore, we also
characterize the graphs $G$ with $rc(G)=m-3$.\\[3mm]
\end{abstract}

\section{Introduction}

All graphs in this paper are finite, undirected and simple. We
follow the terminology and notation of Bondy and Murty \cite{Bondy}.
Let $G$ be a nontrivial connected graph on which is defined a
coloring $c:E(G)\rightarrow \{1,2,\cdots,\ell\}$, $\ell\in
\mathbb{N}$, of the edges of $G$, where adjacent edges may be
colored the same. A path is a $rainbow$ $path$ if no two edges of it
are colored the same. An edge-coloring graph $G$ is $rainbow
~connected$ if any two vertices are connected by a rainbow path.
Clearly, if a graph is rainbow connected, it must be connected.
Conversely, any connected graph has a trivial edge-coloring that
makes it rainbow connected; just color each edge with a distinct
color. Thus, we define the $rainbow~connection~number$ of a
connected graph $G$, denoted by $rc(G)$, as the smallest number of
colors that are needed in order to make $G$ rainbow connected. If
$G_1$ is a connected spanning subgraph of $G$, then $rc(G)\leq
rc(G_1)$. Chartrand et al. obtained that $rc(G)=1$ if and only if
$G$ is complete, and that $rc(G)=m$ if and only if $G$ is a tree, as
well as that a cycle with $k>3$ vertices has rainbow connection
number $\lceil \frac{k}{2}\rceil$, a triangle has rainbow connection
number 1 \cite{Chartrand 1}. Also notice that, clearly, $rc(G)\geq
diam(G)$, where $diam(G)$ denotes the diameter of $G$. For more
information on rainbow connections, we refer to \cite{LiShiSun,
Li-Sun2}. In an edge-colored graph $G$, we use $c(e)$ to denote the
color of edge $e$, then for a subgraph $G_2$ of $G$, $c(G_2)$
denotes the set of colors of edges in $G_2$.

Since $rc(G)=m$ if and only if $G$ is a tree, $rc(G)\neq m-1$ and
$G$ is not a tree if and only if $rc(G)\leq m-2$ (Observation
\ref{ob3}), then there is an interesting problem: Characterize the
graphs with $rc(G)=m-2$. In this paper, we settle down this problem.
Furthermore, we also characterize the graphs $G$ with $rc(G)=m-3$.

We use~$V(G)$, $E(G)$ for the set of vertices and edges of $G$,
respectively. A $pendant~edge$ of $G$ is an edge incident to a
vertex of degree 1. The girth of $G$, denoted by $g(G)$, is the
length of a smallest cycle in $G$. A block of $G$ is a maximal
connected subgraph of $G$ that does not have any cut vertex. So
every block of a nontrivial connected graph is either a $K_2$ or a
2-connected subgraph. All the blocks of a graph $G$ form a block
decomposition of $G$. A $rooted~tree~T(x)$ is a tree $T$ with a
specified vertex $x$, called the $root$ of $T$. Let $L(x)$ denote
the set of leaves of $T(x)$ and $|L(x)|=l(x)$. If $T(x)$ is a
trivial tree, then $l(x)=0$. We let $P_n$ and $C_n$ be the path and
cycle with $n$ vertices, respectively. And $xPy$ denotes the path
from $x$ to $y$.

Let $G$ be a connected graph with $n$ vertices and $m$ edges. Define
the $cyclomatic~number$ of $G$ as $c(G)=m-n+1$. A graph $G$ with
$c(G)=k$ is called a $k$-$cyclic$ graph. According to this
definition, if a graph $G$ meets $c(G)=0$, 1, 2 or 3, then the graph
$G$ is called acyclic (or a tree), unicyclic, bicyclic, or
tricyclic, respectively. Thus the girth of a unicyclic graph is the
unique cycle in the graph. Let $[t]=\{1,\cdots,t\}$ denote the set
of the first $t$ natural numbers. For a set $S$, $|S|$ denotes the
cardinality of $S$.

\section{Some basic results}

We first give an observation which will be useful in the sequel.
\begin{obs}\label{ob1} \cite{Li-Sun1}
If $G$ is a connected graph and $\{E_i\}_{i\in [t]}$ is a partition
of the edge set of $G$ into connected subgraphs $G_i=G[E_i]$, then
$$rc(G)\leq \sum_{i=1}^{t}{rc(G_i)}.$$\qed
\end{obs}

We now give a necessary condition for an edge-colored graph to be
rainbow connected. If $G$ is rainbow connected under some
edge-coloring, then for any two cut edges (if exist) $e_1=u_1u_2$,
$e_2=v_1v_2$, there must exist some $1\leq i,j\leq 2$, such that any
$u_i-v_j$ path must contain edge $e_1,e_2$. So we have:
\begin{obs}\label{ob2}
If $G$ is rainbow connected under some edge-coloring, $e_1$ and
$e_2$ are any two cut edges, then $c(e_1)\neq c(e_2)$.
\end{obs}

For a connected graph $G$, if it is a tree, then $rc(G)=m$; if it contains a
unique cycle of length $k$, then we give the cycle a rainbow
coloring using $\lceil \frac{k}{2}\rceil$ colors (if the cycle is a
triangle, we just need one color) and color each other edge with a
fresh color. Then by Observation \ref{ob1}, we have $rc(G)\leq
(m-k)+\lceil \frac{k}{2}\rceil \leq m-2$. So we have the following
observation.
\begin{obs}\label{ob3}
Let $G$ be a connected graph with $m$ edges. Then $rc(G)\neq m-1$
and $G$ is not a tree if and only if $rc(G)\leq m-2$. Moreover, if
$G$ contains a cycle of length $k$($k\geq4$), then $rc(G)\leq
m-\lfloor\frac{k}{2}\rfloor$.
\end{obs}

For a connected graph $G$, if it contains two edge-disjoint
2-connected subgraphs $B_1$ and $B_2$, then by Observation
\ref{ob3}, we give $B_1$ and $B_2$ a rainbow coloring using
$|E(B_1)|-2$ and $|E(B_2)|-2$ colors, respectively, and color each
other edge with a fresh color. Then by Observation \ref{ob1}, we
have $rc(G)\leq m-4$. So the following lemma holds.

\begin{lem}\label{lem1}
Let $G$ be a connected graph with $m$ edges. If it contains two
edge-disjoint 2-connected subgraphs, then $rc(G)\leq m-4$.
\end{lem}

To $subdivide$ an edge $e$ is to delete $e$, add a new vertex $x$,
and join $x$ to the ends of $e$. Any graph derived from a graph $G$
by a sequence of edge subdivisions is called a $subdivision$ of $G$.
Given a rainbow coloring of $G$, if we subdivide an edge $e=uv$ of
$G$ by $xu$ and $xv$, then we assign $xu$ the same color as $e$ and
assign $xv$ a new color, which also make the subdivision of $G$
rainbow connected. Hence, the following lemma holds.

\begin{lem}\label{lem2}
Let $G$ be a connected graph, and $H$ be a subdivision of $G$. Then
$rc(H)\leq rc(G)+|E(H)|-|E(G)|$.
\end{lem}

The $\Theta$-$graph$ is a graph consisting of three internally
disjoint paths with common end vertices and of lengths $a$, $b$, and
$c$, respectively, such that $a\leq b\leq c$. Then $a+b+c=m$.

\begin{lem}\label{lem3}
Let $G$ be a $\Theta$-graph with $m$ edges. If $m=5$, then
$rc(G)=m-3$; otherwise, $rc(G)\leq m-4$.
\end{lem}

\begin{proof}
Let the three internally disjoint paths be $P_1, P_2, P_3$ with the
common end vertices $u$ and $v$, and the lengths of $P_1, P_2, P_3$
be $a,b,c$, respectively, where $a\leq b\leq c$. If $m=5$, we color
$uP_1v$ with color 1, $uP_2v$ with colors $1,2$, and $uP_3v$ with
colors $2,1$. The resulting coloring makes $G$ rainbow connected.
Thus, $rc(G)\leq m-3$. Since $diam(G)=2$, it follows that
$rc(G)=m-3$. For $m\geq6$, we first consider the graph $\Theta_1$
with $a=1$, $b=2$ and $c=3$. We color $uP_1v$ with color 1, $uP_2v$
with colors $1,1$, and $uP_3v$ with colors $2,1,2$. Next we consider
the graph $\Theta_2$ with $a=2$, $b=2$ and $c=2$. We color $uP_1v$
with colors 1,2, $uP_2v$ with colors $2,1$, and $uP_3v$ with colors
$2,2$. The resulting colorings make $\Theta_1$ and $\Theta_2$
rainbow connected. For a general $\Theta$-graph $G$ with $m\geq6$,
it is a subdivision of $\Theta_1$ or $\Theta_2$, hence by Lemma
\ref{lem2}, $rc(G)\leq m-4$.
\end{proof}

\section{Characterize unicyclic graphs with $rc(G)=m-2$ and $m-3$}

In this section we first give an observation about unicyclic graphs
which will be used frequently. Let $G$ be a connected unicyclic
graph with the unique cycle $C=v_1v_2\cdots v_sv_1$. For brevity,
orient $C$ clockwise. Then $G$ has the structure as follows: a tree,
denoted by $T(v_i)$, is attached at each vertex $v_i$ of $C$. Note
that, $T(v_i)$ may be trivial. Let $i\neq j$. If
$e_i=x_iy_i$($e_j=x_jy_j$) is a pendant edge which belongs to a tree
$T(v_i)$($T(v_j)$). Then there is a unique path
$x_iP_iv_i$($x_jP_jv_j$) from $x_i$($x_j$) to $v_i$($v_j$). Since
$v_i$ and $v_j$ divide $C$ into two segments $v_iCv_j$ and
$v_jCv_i$, there are exactly two paths between $x_i$ and $x_j$ in
$G$. Let $c=\{1,2,\cdots,\ell\}$ be an edge coloring of $G$. Since
each edge in $G\setminus E(C)$ is a cut edge, by Observation
\ref{ob2}, they must obtain distinct colors. It is easy to see that
$|c(x_iP_iv_i)\cap c(C)|\leq1$. In the process of coloring, we
always first color $G\setminus E(C)$ with $[t]$ colors, then color
$C$, where $t=|E(G)\setminus E(C)|$. Thus, after coloring
$E(G)\setminus E(C)$, the unique path $x_iP_iv_i$ can be viewed as a
pendant edge and every $T(v_i)$ will be a star with the center
vertex $v_i$. Suppose $|c(x_iP_iv_i)\cap c(C)|=1$ and
$|c(x_jP_jv_j)\cap c(C)|=1$, then we can adjust the colors of cut
edges such that $c(e_i)=1$ and $c(e_j)=2$. Thus, $1,2\in v_iCv_j$ or
$1,2\in v_jCv_i$, namely, 1,2 can only be assigned in the same path
from $v_i$ to $v_j$. Moreover, another path from $v_i$ to $v_j$
should be rainbow. We summarize the above argument into an
observation.

\begin{obs}\label{ob4}
Let $G$ be a connected unicyclic graph with the unique cycle
$C=v_1v_2\cdots v_sv_1$, and let $c=\{1,2,\cdots,\ell\}$ be an edge
coloring of $G$. Let $p\in T(v_i)$ and $q,r\in T(v_j)$.

($i$) If $p,q\in C$, then they are in the same path from $v_i$ to
$v_j$ and the other path from $v_i$ to $v_j$ should be rainbow.

($ii$) If $q,r$ are in the unique path from a vertex $x$ of
$V(G)\setminus V(C)$ to $v_j$, then $q$ and $r$ can not both belong
to $C$.
\end{obs}

In this section we only deal with unicyclic graphs. According to the
girth of $G$, we introduce some graph classes and discuss them by
some lemmas. Noth that, $l(v_i)$ is the number of leaves of the tree
attached at the vertex $v_i$ from the unique cycle of $G$.

Let $i$ be an integer with $1\leq i\leq 3$ and the addition is
performed modulo 3. Let $\mathcal{G}=\{G: m=n, g(G)=3\}$,
$\mathcal{G}_1=\{G: G\in \mathcal{G}, l(v_i)\geq1, l(v_{i+1})\geq1,
l(v_{i+2})\geq1,~or~l(v_i)\geq3\}$, $\mathcal{G}_2=\{G: G\in
\mathcal{G}, l(v_i)=0, l(v_{i+1})\leq2, l(v_{i+2})\leq2\}$.
Obviously, $\mathcal{G}=\mathcal{G}_1\cup \mathcal{G}_2$.

\begin{lem}\label{lem4}
Let $G$ be a graph belonging to $\mathcal{G}$. If $G\in
\mathcal{G}_1$, then $rc(G)=m-3$; otherwise $rc(G)=m-2$.
\end{lem}

\begin{proof}
Let the unique cycle of $G$ be $C=v_1v_2v_3v_1$. Suppose $G\in
\mathcal{G}_1$, by Observation \ref{ob2}, each edge of $G\setminus
E(C)$ must obtain a distinct color, color them with a set $[m-3]$ of
colors. We consider two cases. Without loss of generality, first
suppose that $e_i=x_iy_i$ is a pendant edge in $T(v_i)$ that is
assigned color $i$, where $1\leq i\leq 3$. Set $c(v_1v_2)=3$,
$c(v_2v_3)=1$, $c(v_3v_1)=2$. Next suppose that $e_j=x_jy_j$ is a
pendant edge of $T(v_1)$ that is assigned color $j$, where $1\leq
j\leq 3$. Color $E(C)$ with 1,2,3, respectively. It is easy to show
that these two colorings are rainbow, and in these two cases,
$rc(G)=m-3$.

If $G\in \mathcal{G}_2$, by Observation \ref{ob3}, $rc(G)\leq m-2$.
By Observation \ref{ob4}, we know that at most two colors for
$G\setminus E(C)$ can assign to $C$. Thus, we need a fresh color for
$C$, it follows that $rc(G)\geq m-2$. Therefore, $rc(G)=m-2$.
\end{proof}

Let $i$ be an integer with $1\leq i\leq 4$ and the addition is
performed modulo 4. Set $\mathcal{H}=\{G: m=n, g(G)=4\}$ and
$\mathcal{H}=\mathcal{H}_1\cup \mathcal{H}_2\cup \mathcal{H}_3$,
where $\mathcal{H}_2=\{G: G\in \mathcal{H}, l(v_i)=l(v_{i+2})=0,
l(v_{i+1})\leq1, l(v_{i+3})\leq1\}$, $\mathcal{H}_3=\{G: G\in
\mathcal{H}, l(v_i)\geq4, or~ l(v_i)\geq1, l(v_{i+1})\geq2,
l(v_{i+2})\geq1\}$.

\begin{lem}\label{lem5}
Let $G$ be a graph belonging to $\mathcal{H}$. If $G\in
\mathcal{H}_2$, then $rc(G)=m-2$; if $G\in \mathcal{H}_3$, then
$rc(G)=m-4$; if $G\in \mathcal{H}_1$, then $rc(G)=m-3$.
\end{lem}

\begin{proof}
Let the unique cycle of $G$ be $C=v_1v_2v_3v_4v_1$. By Observation
\ref{ob2}, each edge of $G\setminus E(C)$ must obtain a distinct
color, this costs $m-4$ colors, thus $rc(G)\geq m-4$. Color
$G\setminus E(C)$ with a set $[m-4]$ of colors. Suppose $G\in
\mathcal{H}_2$. By Observation \ref{ob3}, $rc(G)\leq m-2$. By
Observation \ref{ob4}, we know that at least two colors different
from $c(G\setminus E(C))$ should assign to $C$, it follows that
$rc(G)\geq m-2$. Hence, $rc(G)=m-2$.

Suppose $G\in \mathcal{H}_3$. First let $e_i=x_iy_i$ be a pendant
edge in $T(v_1)$ that is assigned color $i$, where $1\leq i\leq 4$.
Color $E(C)$ with 1,2,3,4, respectively. Next suppose that
$e_j=x_jy_j$ is a pendant edge that is assigned color $j$ such that
$1\in T(v_1)$, $2,3\in T(v_2)$ and $4\in T(v_3)$, where $1\leq j\leq
4$. Set $c(v_1v_2)=4$, $c(v_2v_3)=1$, $c(v_3v_4)=3$, $c(v_1v_4)=2$.
It is easy to show that these two colorings are rainbow, and in
these two cases, $rc(G)=m-4$.

If $G\in \mathcal{H}_1$, by Observation \ref{ob4}, we check one by
one that at least one color different from $c(G\setminus E(C))$
should assign to $C$, thus $rc(G)\geq m-3$. If $e_1$ and $e_2$ are
two pendant edges in a tree (say $T(v_1)$) that are assigned colors
1 and 2, respectively. Set $c(v_1v_2)=m-3$, $c(v_2v_3)=1$,
$c(v_3v_4)=2$, $c(v_1v_4)=m-3$. By symmetry, it remains to consider
the case that $l(v_1)=l(v_2)=l(v_3)=1$. Suppose that $e_i=x_iy_i$ is
a pendant edge in $T(v_i)$ that is assigned color $i$, where $1\leq
i\leq 3$. Set $c(v_1v_2)=3$, $c(v_2v_3)=1$, $c(v_3v_4)=m-3$,
$c(v_1v_4)=2$. It is easy to show that these two colorings are
rainbow, and in these two cases, $rc(G)=m-3$.
\end{proof}

Let $i$ be an integer with $1\leq i\leq 5$ and the addition is
performed modulo 5. Set $\mathcal{J}=\{G: m=n, g(G)=5\}$ and
$\mathcal{J}=\mathcal{J}_1\cup \mathcal{J}_2\cup C_5$, where
$\mathcal{J}_1=\{G: G\in \mathcal{J}, l(v_i)\leq2, l(v_{i+2})\leq1,
l(v_{i+1})=l(v_{i+3})=l(v_{i+4})=0~ or ~l(v_i)\leq1,
l(v_{i+1})\leq1, l(v_{i+2})\leq1, l(v_{i+3})=l(v_{i+4})=0\}$.

\begin{lem}\label{lem6}
Let $G$ be a graph belonging to $\mathcal{J}$. If $G$ is isomorphic
to a cycle $C_5$, then $rc(G)=m-2$. If $G\in \mathcal{J}_1$, then
$rc(G)=m-3$. If $G\in \mathcal{J}_2$, then $rc(G)\leq m-4$.
\end{lem}

\begin{proof}
Let the unique cycle of $G$ be $C=v_1v_2v_3v_4v_5v_1$. If $G$ is
isomorphic to a cycle $C_5$, it is easy to see that $rc(G)=m-2$.
Suppose $G\in \mathcal{J}_1$. Suppose $e_1$ is a pendant edge of
$T(v_1)$ that is assigned color 1. Set $c(v_1v_2)=m-4$,
$c(v_2v_3)=m-3$, $c(v_3v_4)=1$, $c(v_4v_5)=m-4$, $c(v_1v_5)=m-3$.
Thus $rc(G)\leq m-3$. On the other hand, since it costs $m-5$ colors
for $G\setminus E(C)$,  and by Observation \ref{ob4}, we know that
at least two colors different from $c(G\setminus E(C))$ should
assign to $C$, it follows that $rc(G)\geq m-3$. Therefore,
$rc(G)=m-3$.

Suppose $G\in \mathcal{J}_2$. Without loss of generality, we
consider the following three cases. If $l(v_i)\geq3$ for some $i$
with $1\leq i\leq 5$, then we may suppose that $e_1$, $e_2$ and
$e_3$ are the three pendant edges of $T(v_1)$ that are assigned
colors 1,2,3, respectively. Set $c(v_1v_2)=m-4$, $c(v_2v_3)=3$,
$c(v_3v_4)=2$, $c(v_4v_5)=1$, $c(v_1v_5)=m-4$. If $l(v_i)=2$, then
we may suppose that $e_1$, $e_2$ are the two pendant edges of
$T(v_1)$ that are assigned colors 1,2, respectively, and $e_3$ is a
pendant edge of $T(v_2)$ that is assigned color 3. Set
$c(v_1v_2)=m-4$, $c(v_2v_3)=1$, $c(v_3v_4)=2$, $c(v_4v_5)=m-4$,
$c(v_1v_5)=3$. It remains to consider the case that $l(v_i)\leq1$
for each $i$. Without loss of generality, let
$l(v_1)=l(v_2)=l(v_4)=1$. Suppose that $e_i$ is a pendant edge that
is assigned color $i$ such that $e_1\in T(v_1)$, $e_2\in T(v_2)$ and
$e_3\in T(v_4)$, where $1\leq i\leq 3$. Set $c(v_1v_2)=3$,
$c(v_2v_3)=m-4$, $c(v_3v_4)=1$, $c(v_4v_5)=2$, $c(v_1v_5)=m-4$. It
is easy to show that these three colorings are rainbow, and in these
three cases, $rc(G)\leq m-4$.
\end{proof}

Let $i$ be an integer with $1\leq i\leq 6$ and the addition is
performed modulo 6. Set $\mathcal{L}=\{G: m=n, g(G)=6\}$ and
$\mathcal{L}=\mathcal{L}_1\cup \mathcal{L}_2$, where
$\mathcal{L}_1=\{G: G\in \mathcal{L}, l(v_i)\leq1, l(v_{i+3})\leq1,
l(v_{i+1})=l(v_{i+2})=l(v_{i+4})=l(v_{i+5})=0\}$.

\begin{lem}\label{lem7}
Let $G$ be a graph belonging to $\mathcal{L}$. If $G\in
\mathcal{L}_1$, then $rc(G)=m-3$; otherwise $rc(G)\leq m-4$.
\end{lem}

\begin{proof}
Let the unique cycle of $G$ be $C=v_1v_2v_3v_4v_5v_6v_1$. By
Observation \ref{ob2}, each edge of $G\setminus E(C)$ must obtain a
distinct color, this costs $m-6$ colors, thus $rc(G)\geq m-6$. Color
$G\setminus E(C)$ with a set $[m-6]$ of colors. Suppose $G\in
\mathcal{L}_1$. Set $c(v_1v_2)=m-5$, $c(v_2v_3)=m-4$,
$c(v_3v_4)=m-3$, $c(v_4v_5)=m-5$, $c(v_5v_6)=m-4$, $c(v_1v_6)=m-3$.
By Observation \ref{ob2}, $rc(G)\leq m-3$. On the other hand, by
Observation \ref{ob4}, we know that at least three colors different
from $c(G\setminus E(C))$ should assign to $C$, it follows that
$rc(G)\geq m-3$. Therefore, $rc(G)=m-3$.

Suppose $G\in \mathcal{L}_2$. If $l(v_i)\geq2$, then we may suppose
that $e_1$ and $e_2$ are the two pendant edges of $T(v_1)$ that are
assigned colors 1,2, respectively. Set $c(v_1v_2)=m-5$,
$c(v_2v_3)=m-4$, $c(v_3v_4)=1$, $c(v_4v_5)=2$, $c(v_5v_6)=m-5$,
$c(v_1v_6)=m-4$. It remains to consider the case that $l(v_i)\leq1$
for each $i$. Suppose $l(v_1)=l(v_2)=1$. Let $e_1$ and $e_2$ be the
two pendant edges that are assigned colors 1,2, respectively, such
that $e_1\in T(v_1)$ and $e_2\in T(v_2)$. Set $c(v_1v_2)=m-5$,
$c(v_2v_3)=m-4$, $c(v_3v_4)=1$, $c(v_4v_5)=2$, $c(v_5v_6)=m-5$,
$c(v_1v_6)=m-4$. Without loss of generality, let $l(v_1)=l(v_3)=1$.
Suppose that $e_1$ and $e_2$ are the two pendant edges that are
assigned colors 1,2, respectively, such that $e_1\in T(v_1)$ and
$e_2\in T(v_3)$. Set $c(v_1v_2)=m-5$, $c(v_2v_3)=m-4$,
$c(v_3v_4)=1$, $c(v_4v_5)=m-5$, $c(v_5v_6)=m-4$, $c(v_1v_6)=2$. It
is easy to show that these three colorings are rainbow, and in these
three cases, $rc(G)\leq m-4$.
\end{proof}

\section{Characterize graphs with $rc(G)=m-2$ and $m-3$}

Now we are ready to characterize the graphs with $rc(G)=m-2$ and $rc(G)=m-3$.

\begin{thm}\label{thm1}
$rc(G)=m-2$ if and only if $G$ is isomorphic to a cycle $C_5$ or
belongs to $\mathcal{G}_2\cup\mathcal{H}_2$.
\end{thm}
\begin{proof}
Suppose that $G$ is a graph with $rc(G)=m-2$. By Lemma \ref{lem1},
$G$ contains a unique 2-connected subgraph. By Lemma \ref{lem2}, $G$
contains no $\Theta$-graph as a subgraph. It follows that $G$ is a
unicyclic graph. By Observation \ref{ob3}, the girth of $G$ is at
most 5. The cases that the girth of $G$ is 3,4 and 5 have been
discussed in Lemmas \ref{lem4}, \ref{lem5} and \ref{lem6},
respectively. We conclude that $G$ must be isomorphic to a graph
shown in our theorem.

Conversely, By Lemmas \ref{lem4}, \ref{lem5} and \ref{lem6}, the
result holds.
\end{proof}

Let $\mathcal{M}$ be a class of graphs where in each graph a path is
attached at each vertex of degree 2 of $K_4-e$, respectively. Note
that, the path may be trivial.

\begin{thm}\label{thm2}
$rc(G)=m-3$ if and only if $G$ is isomorphic to a cycle $C_7$ or
belongs to
$\mathcal{G}_1\cup\mathcal{H}_1\cup\mathcal{J}_1\cup\mathcal{L}_1\cup\mathcal{M}$.
\end{thm}

\begin{proof}
Suppose that $G$ is a graph with $rc(G)=m-3$. By Lemma \ref{lem1},
$G$ contains a unique 2-connected subgraph $B$. Set
$V(B)=\{v_1,\cdots,v_s\}$, then $G$ has the structure as follows: a
tree, denoted by $T(v_i)$, is attached at each vertex $v_i$ of $B$.
If $B$ is exactly a cycle, then by Observation \ref{ob3}, the girth
of $G$ is at most 7. The cases that the girth of $G$ is 3,4,5 and 6
have been discussed in Lemmas \ref{lem4}, \ref{lem5}, \ref{lem6} and
\ref{lem7}, respectively. It remains to deal with the case that the
girth of $G$ is 7. If $G$ is not isomorphic to a cycle $C_7$, then
suppose that $e_1$ is a pendant edge of $T(v_1)$ that is assigned
color 1. Color $G\setminus E(B)$ with a set $[m-7]$ of colors and
set $c(v_1v_2)=m-6$, $c(v_2v_3)=m-5$, $c(v_3v_4)=m-4$,
$c(v_4v_5)=1$, $c(v_5v_6)=m-6$, $c(v_6v_7)=m-5$, $c(v_1v_7)=m-4$. By
Observation \ref{ob1}, we have $rc(G)\leq m-4$.

So $B$ is not a cycle. By Lemma \ref{lem3}, $G$ contains no
$\Theta$-graph except a $K_4-e$ as a subgraph. We first claim that
$B$ is isomorphic to a $K_4-e$. If $B$ is isomorphic to a $K_4$, we
first color the edges of $G\setminus E(B)$ with $m-6$ colors, then
give each edge of $B$ the same new color, this costs $m-5$ colors
totally, it is easy to check that this coloring is rainbow, and in
this case, $rc(G)\leq m-5$, a contradiction.  Set
$V(K_4-e)=\{v_1,v_2,v_3,v_4\}$, and
$E(K_4-e)=\{v_1v_2,v_2v_3,v_3v_4,v_4v_1,v_1v_3\}$. If $G\notin
\mathcal{M}$, then $l(v_i)\geq1$ or $l(v_j)\geq2$ where $i=1~or~3$,
$j=2~or~4$. If $l(v_1)\geq1$, suppose that $e_1$ is a pendant edge
of $T(v_1)$ that is assigned color 1. Assign color 1 to $v_2v_3$ and
$m-4$ to each other edge of $K_4-e$. If $l(v_2)\geq2$, suppose that
$e_1$ and $e_2$ are two pendant edges of $T(v_2)$ that are assigned
colors 1 and 2, respectively. Set
$c(v_1v_2)=c(v_2v_3)=c(v_1v_3)=m-4$, $c(v_3v_4)=1$, $c(v_1v_4)=2$.
In both cases, $rc(G)\leq m-4$. We conclude that $G$ must be
isomorphic to a graph shown in our theorem.

Conversely, if $G$ is isomorphic to a cycle $C_7$, then $rc(G)=m-3$.
If $G\in \mathcal{M}$, it is easy to see that at least two new
colors different from $c(G\setminus E(B))$ should be assigned to
$B$. Since each edge of $G\setminus E(B)$ must obtain a distinct
color, this costs $m-5$ colors, it follows that $rc(G)\geq m-3$. Set
$c(v_1v_2)=c(v_3v_4)=c(v_1v_3)=m-4$, $c(v_2v_3)=c(v_1v_4)=m-3$, thus
$rc(G)\leq m-3$. Therefore, $rc(G)=m-3$. By Lemmas \ref{lem4},
\ref{lem5}, \ref{lem6} and \ref{lem7}, the result holds.
\end{proof}

\end{document}